\documentclass[12pt,reqno]{amsart}
\usepackage{fullpage}
\usepackage[top=1.5in, bottom=2in, left=2.5in, right=2.5in]{geometry}
\usepackage{amssymb,amsmath,amsthm}
\usepackage{comment,enumitem,url}
\usepackage{xcolor}
\usepackage{enumitem}
\usepackage{bm}
\usepackage{float}

\usepackage{graphicx}
\usepackage[draft]{hyperref}
\usepackage{amsmath,amsopn,amssymb,amsfonts}
\usepackage{verbatim}
\usepackage{amsthm}
\usepackage{mathtools}
\usepackage{color}
\usepackage{enumitem}
\usepackage[framemethod=TikZ]{mdframed}
\usepackage{bbm}
\usepackage{mathrsfs}
\usepackage{booktabs}
\usepackage{caption}
\usepackage{cancel}
\usepackage{cleveref,tikz-cd}

\renewcommand{\H}{\mathbb{H}}

\newcommand{\sgn}{\mbox{sgn}}

\newcommand{\SL}{\mathrm{SL}}

\newcommand{\N}{\mathbb N}
\newcommand{\C}{\mathbb C}

\theoremstyle{plain}
\newtheorem{thm}{Theorem}[section]

\newtheorem{lem}[thm]{Lemma}

\newtheorem*{rem}{Remark}

\theoremstyle{definition}

\newtheorem*{defn*}{Definition}

\numberwithin{equation}{section}
\numberwithin{thm}{section}

\setlist[enumerate]{leftmargin=*,label=\rm{(\arabic*)}}

\renewcommand{\sgn}{\textnormal{sgn}}

\def\del{ \partial}

\newcommand{\re}{{\rm Re}}

\renewcommand{\sgn}{{\rm sgn}}
\newcommand{\R}{\mathbb R}
\newcommand{\Z}{\mathbb Z}

\setlist[itemize]{noitemsep, topsep=0pt}

\allowdisplaybreaks

\makeatletter
\newcommand{\vast}{\bBigg@{4}}
\newcommand{\Vast}{\bBigg@{5}}
\makeatother

\renewcommand{\pmod}[1]{\ \left( \mathrm{mod} \, #1 \right)}

\title{A direct proof of Mono--Rolen--Stumpenhusen and new constructions via the Maass raising operators}

\author{Kathrin Bringmann}
\address{Department of Mathematics and Computer Science, Division of Mathematics, University of Cologne, Weyertal 86--90, 50931 Cologne, Germany}
\email{kbringma@math.uni-koeln.de}
\author{Ben Kane}
\address{The University of Hong Kong, Department of Mathematics, Pokfulam, Hong Kong}
\email{bkane@hku.hk}

\date{\today}

\makeatletter
\@namedef{subjclassname@2020}{\textup{2020} Mathematics Subject Classification}
\makeatother



\begin{document}
	
	\begin{abstract}
		In this paper, we give a direct conceptual proof of the main result of Mono, Rolen, and Stumpenhusen, using differential operators. More precisely, we realize their functions $\omega_{k+1,D}$ as images of the quadratic form Poincar\'e   series $f_{k,D}$ under the Maass raising operator. This perspective gives a natural explanation for the modularity and Laplace eigenvalue properties of $\omega_{k+1,D}$. We further extend these results by investigating the images of more general local Maass forms under the Maass raising and lowering operators.
	\end{abstract}
	\maketitle
	
	\section{Introduction and statements of results}
Recently, Mono, Rolen, and Stumpenhusen \cite{MRS} studied (for $k\in \mathbb{N}$)
	\begin{equation*}
		\omega_{k+1,D}(z) := \sum_{Q\in \mathcal Q_D} \frac{Q_z}{Q(z,1)^{k+1}}.
	\end{equation*}
Here, for $D>0$ a discriminant, $\mathcal Q_D$ denotes the set of integral binary quadratic form of discriminant $D$ and for $Q=[a,b,c]\in \mathcal Q_D$ ($z=x+iy$ with $x\in \R, y\in \R^{+}$ throughout), we let
	\begin{equation*}
		Q_z := \frac{1}{y} \left(a|z|^2+bx+c\right).
	\end{equation*}

These functions may be viewed as non-holomorphic analogues of the ``quadratic form Poincar\'e series" 
\begin{equation}\label{fk}
f_{k,D}(z):=\sum_{Q\in\mathcal{Q}_D} \frac{1}{Q(z,1)^k}.
\end{equation}
The functions $f_{k,D}$ were discovered by Zagier \cite{ZagierDoiNaganuma} while investigating the Doi--Naganuma lift. \!Kohnen and Zagier \cite{KohnenZagier} then showed that their generating function is the kernel function for the Shimura and Shintani lifts.

These kernel functions have further applications to central values of twists of $L$-functions.  To state the results from \cite{MRS}, let $\widehat{E}_2(z):=E_2(z)-\frac{3}{\pi y}$ denote the modular completion of the weight two Eisenstein series. For $\kappa\in\N_0 +\frac{1}{2}$ and $m\in\N$, denote the $m$-th weight Poincar\'e series in Kohnen's plus space on $\Gamma_0(4)$ by $P_{\kappa,m}^+$ (see \eqref{eqn:Poincaredef}). Following \cite{BKO}, we furthermore define the \begin{it}divisor modular form\end{it} for a weight $k$ meromorphic modular form $f$ by 
\[
f^{\mathrm{div}}(z):=\sum_{\tau\in\SL_2(\Z)\backslash\H} e_\tau H_{\tau}(z),
\]
where $e_\tau := \frac{1}{2}$ if $\tau=i$, $e_\tau:=\frac{1}{3}$ if $\tau=\rho := e^{\frac{\pi i}{3}}$, and $e_\tau:=1$ otherwise.
Moreover $H_{\tau}(z):=\frac{E_4^2(z)E_6(z)}{\Delta(z)} \frac{1}{j(z)-j(\tau)}$ with $j$ the modular $j$-invariant, $\Delta$ the modular discriminant, and $E_k$ the weight $k$ Eisenstein series.  Recall that weak Maass forms transform like modular forms but instead of being holomorphic they are eigenfunctions under the weight $\kappa$ Laplace operator $\Delta_\kappa$ (see Section \ref{sec:ModularForms}). The following properties of $\omega_{k+1,D}$ were proven in \cite[Theorem 1.1]{MRS}.
\begin{thm}\label{thm:MRS}
Let $k\ge2$.
\begin{enumerate}
\item The function $\omega_{k+1,D}$ is a weak Maass form of weight $2k+2$ with eigenvalue $2k$ under $\Delta_{2k+2}$. 
\item  We have the splitting, with $Q'(z, 1):=\frac{\del}{\del z} Q(z,1)$,
		\begin{equation*}
			\omega_{k+1,D}(z) = \frac{f_{k,D}(z)}{y} -i\sum_{Q\in\mathcal Q_D} \frac{Q'(z,1)}{Q(z,1)^{k+1}}.
		\end{equation*}
\item  We have
		\begin{equation*}
			\omega_{k+1,D} =\frac{2\pi}{k} f_{k,D}^{\mathrm{div}} f_{k,D} -\frac{\pi}{3} \widehat E_2f_{k,D}.
		\end{equation*} 
\item We have
		$$
		 \omega_{k+1,D}(z)=\frac{6(4\pi)^{k-\frac{1}{2}}}{\Gamma\left(k-\frac{1}{2}\right)}\left\langle P_{k+\frac{1}{2},D}^+,\Lambda_k(\,\cdot\,,-\overline z)\right\rangle,
		$$	
where $\langle \cdot,\cdot\rangle$ is the Petersson inner product defined in \eqref{Pet} and 
\[
\Lambda_k(\tau,z):=\sum_{D\ge 1}D^{k-\frac{1}{2}}	\omega_{k+1,D}(z) e^{2\pi iD\tau}
\]
satisfies weight $k+\frac{1}{2}$ modularity on $\Gamma_0(4)$ in $\tau$. 
\end{enumerate}
\end{thm}
In this paper, we give a conceptual explanation of Theorem \ref{thm:MRS} by relating $\omega_{k+1,D}$ and $f_{k,D}$ via the \begin{it}raising operator\end{it} $R_\ell := 2i\frac{\partial}{\partial z} + \frac{\ell}{y}$, which maps modular objects of weight $\ell$ to those of weight $\ell+2$. Our main theorem is the following simple identity.
\begin{thm}\label{prop:fkDraise}
For $k\geq 2$ and $D\in\Z$, we have 
\[
\omega_{k+1,D} = \frac{R_{2k}\left(f_{k,D}\right)}{2k}.
\]
\end{thm}
\begin{rem} 
	Since the Maass raising operator intertwines with the slash action and the action of the Laplace operator,  Theorem \ref{thm:MRS} $(1)$ follows immediately from Theorem \ref{prop:fkDraise}.
\end{rem}

Following \cite[(1.4)]{BKK}, we define 
\[
\mathcal{F}_{1-k,D}(z):=\frac{D^{\frac{1}{2}-k}}{\binom{2k-2}{k-1}\pi}\sum_{Q=[a,b,c]\in\mathcal{Q}_D} \sgn\left(Q_z\right)Q(z,1)^{k-1}\psi\left(\frac{Dy^2}{|Q(z,1)|^2}\right),
\]
where $\psi(v):=\frac{1}{2}\beta(v;k-\frac{1}{2},\frac{1}{2})$ with $\beta(v;s,w):=\int_{0}^{v} u^{s-1} (1-u)^{w-1} \quad (\re(w)>0)$ the \begin{it}incomplete beta function\end{it}. In \cite[Theorem 1.2]{BKK} it was shown that $\mathcal{F}_{1-k,D}$ is a weight $2-2k$ analogue of $f_{k,D}$ in the sense that it is almost everywhere a preimage under $\xi_{2-2k}$, where $\xi_{\kappa}:=2iy^{\kappa}\overline{\frac{\partial}{\partial\overline{z}}}$ is the $\xi$\begin{it}-operator\end{it}. In particular, for $z$ outside of the exceptional set 
\begin{equation*}
E_D:=\bigcup_{Q\in\mathcal{Q}_D}\left\{z\in\H: Q_z=0\right\}
\end{equation*}
we have\footnote{Note the different normalization of $f_{k,D}$ in \cite[(1.1)]{BKK}.}
\begin{equation}\label{Ver}
\xi_{2-2k}\left(\mathcal{F}_{1-k,D}(z)\right)= \frac{f_{k,D}(z)}{\binom{2k-2}{k-1}\pi}.
\end{equation}
Next define
\begin{equation}\label{eqn:Wdef}
\mathcal{W}_{-k,D}(z):=y^{2k}\overline{f_{k,D}(z)}.
\end{equation}
As we show in the next theorem, $\mathcal{W}_{-k,D}$ is a weight $-2k$ analogue of $\omega_{k+1,D}$: it is a preimage under $\xi_{-2k}$ and is related to $\mathcal{F}_{1-k,D}$ by $L_{2-2k}$, where $L_\kappa := -2iy^2\frac{\partial}{\partial \overline z}$ is the classical Maass \textit{lowering operator}.
\begin{thm}\label{prop:omegaNegProp}\hspace{0cm}

\begin{enumerate}[leftmargin=*]
\item
The function $\mathcal{W}_{-k,D}$ is a weight $-2k$ weak Maass form with eigenvalue $2k$. For $z\notin E_D$, we have
\[
\mathcal{W}_{-k,D}(z)=\binom{2k-2}{k-1}\pi L_{2-2k}\left(\mathcal{F}_{1-k,D}(z)\right).
\] 

\item
We have 
\[
\xi_{-2k}\left(\mathcal{W}_{-k,D}\right)=2k\omega_{k+1,D}.
\]
\end{enumerate}
\end{thm}
\begin{rem}
Theorem \ref{prop:omegaNegProp} \textup{(2)} implicitly appears as a termwise calculation in \cite[(3.2)]{MRS}. 
\end{rem}
Ignoring multiplicative constants, Theorem \ref{prop:fkDraise} with Theorem \ref{prop:omegaNegProp} can be combined to obtain the following commutative diagram outside $E_D$: 
\begin{figure}[htbp] \label{fig:figure1}
    \begin{center}
    \begin{tikzcd}[column sep=0.65in,row sep=0.65in]
        \mathcal{F}_{1-k,D} \arrow[d, "\xi_{2-2k}"] \arrow[rr, "L_{2-2k}"] & & y^{2k}\overline{f_{k,D}} \arrow[d, bend left=20, "\xi_{-2k}"] \\
        f_{k,D} \arrow[rr, "R_{2k}"]& & \omega_{k+1,D} \arrow[u, bend left=20, "\xi_{2k+2}"]
    \end{tikzcd}
    \end{center}
\end{figure}

In \cite[Theorem 1.1 (1)]{BKV}, $f_{k,D}$ was embedded into a family of local Maass forms $f_{k,s,D}$ (see \eqref{eqn:fksDdef} for the explicit definition) depending on an additional spectral parameter. Outside of the exceptional set $E_D$ where $f_{k,s,D}$ may exhibit discontinuities, these functions are eigenfunctions of $\Delta_{2k}$. In light of Theorem \ref{prop:fkDraise} and the fact that 
\begin{equation}\label{eqn:DeltaLR}
	\Delta_{2k}=-R_{2k-2} \circ L_{2k}=-L_{2k+2} \circ R_{2k}-2k,
\end{equation}
it is natural to investigate the action of $R_{2k}$ on $f_{k,s,D}$. Another local Maass form $\mathcal{F}_{k,s,D}$ (see \eqref{eqn:calFsDdef} for the explicit definition) generalizing $\mathcal{F}_{1-k,D}$, which was also investigated in \cite[Theorem 1.3]{BKV}, naturally arises by considering raising and lowering on $f_{k,s,D}$. 
\begin{thm}\label{thm:raiselowersparam}
Suppose that $k\in\Z\setminus\{0,1\}$.
\noindent

\noindent
\begin{enumerate}[leftmargin=*]
\item For $\re(s)\gg1$ and $z\notin E_D$, we have  
\begin{align*}
R_{2k}(f_{k,s,D}(z))&= \frac{2\Gamma\left(s+\frac{k}{2}+\frac{3}{4}\right)}{(4\pi)^{k}\Gamma\left(s-\frac{k}{2}-\frac{1}{4}\right)}\mathcal{F}_{k+1,s,D}(z),\\
R_{2k}(\mathcal{F}_{k,s,D}(z))&= \frac{2(4\pi)^k\left(s+\frac{k}{2}-\frac{1}{4}\right)\Gamma\left(s-\frac{k}{2}+\frac{1}{4}\right)}{\Gamma\left(s+\frac{k}{2}+\frac{1}{4}\right)} f_{k+1,s,D}(z),
\end{align*} where $\Gamma(\cdot)$ denotes the usual Gamma function. 
\item For $\re(s)\gg1$ and $z\notin E_D$, we have 
\begin{align*}
L_{2k}(f_{k,s,D}(z))&=\frac{2\Gamma\left(s+\frac{k}{2}-\frac{1}{4}\right)}{(4\pi)^{k-1}\Gamma\left(s-\frac{k}{2}-\frac{1}{4}\right)}\mathcal{F}_{k-1,s,D}(z),\\
 L_{2k}(\mathcal{F}_{k,s,D}(z))&= \frac{2(4\pi)^{k-1}\left(s-\frac{k}{2} -\frac{1}{4}\right)\Gamma\left(s-\frac{k}{2}+\frac{1}{4}\right)}{\Gamma\left(s+\frac{k}{2}-\frac{3}{4}\right)} f_{k-1,s,D}(z).
\end{align*}
\end{enumerate}
\end{thm}
We next realize $\omega_{k+1,D}$ as a limiting case of $\mathcal F_{k+1,s,D}$.
\begin{thm}\label{prop:omegasparam}
We have
\[
\omega_{k+1,D}=\frac{24\pi\sqrt{D}(2k-1)!}{4^{k}(k-1)!(4\pi D)^{\frac{k}{2}+\frac{3}{4}}}\lim_{s\to \frac{k}{2}+\frac{1}{4}} \frac{\mathcal{F}_{k+1,s,D}}{\Gamma\left(s-\frac{k}{2}-\frac{1}{4}\right)}.
\]
\end{thm}

The paper is organized as follows. In Section \ref{sec:prelim} we recall basic facts required for this paper. In Section \ref{sec:reproof} we prove Theorems \ref{thm:MRS}, \ref{prop:fkDraise}, and \ref{prop:omegaNegProp} and Section \ref{sec:sparam} is devoted to the proof of Theorems \ref{thm:raiselowersparam} and \ref{prop:omegasparam}.

	\section*{Acknowledgement}
	The first author has received funding from the European Research Council (ERC) under the European Union’s Horizon 2020 research and innovation programme (grant agreement No. 101001179). The second author was supported by grants from the Research Grants Council of the Hong Kong SAR, China (project numbers HKU 17314122 and 17305923). The authors thank Andreas Mono for helpful comments.

\section{Preliminaries}\label{sec:prelim}
\subsection{Holomorphic and non-holomorphic modular forms}\label{sec:ModularForms}

We briefly recall basic facts on modular forms and Maass forms needed below. For $\kappa\in\frac{1}{2}\Z$ and $\gamma=\left(\begin{smallmatrix}a&b\\c&d\end{smallmatrix}\right)\in \Gamma\subseteq\SL_2(\Z)$ (with $\Gamma\subseteq\Gamma_0(4)$ if $\kappa\notin\Z$), we define the weight $\kappa$ \begin{it}slash operator\end{it}
\[
f|_{\kappa}\gamma(z):=\begin{cases}
(cz+d)^{-\kappa}f(\gamma z)&\text{if }\kappa\in\Z,\\ 
\left(\frac{c}{d}\right) \varepsilon_d^{2\kappa } (cz+d)^{-\kappa} f(\gamma z)&\text{if }\kappa\in\Z+\frac{1}{2},
\end{cases}
\]
where $(\frac{\,\cdot\,}{\cdot})$ denotes the Jacobi symbol and, for $d$ odd, we let 
\[
\varepsilon_d:=\begin{cases}1&\text{if }d\equiv 1\pmod{4},\\ i&\text{if }d\equiv 3\pmod{4}.\end{cases}
\]
A function $f$ is \emph{modular of weight $\kappa$ } on $\Gamma$ if $f|_{\kappa}\gamma=f$ for every $\gamma\in\Gamma$. If $f$ is holomorphic on $\H$, satisfies weight $\kappa$ modularity on $\Gamma$, and $f|_{\kappa}\gamma(z)=O(1)$ as $z\to i\infty$ for all $\gamma\in\SL_2(\Z)$, then we call $f$ a \begin{it}(holomorphic) modular form of weight $\kappa$ on $\Gamma$\end{it}. The space of such forms is denoted by $M_{\kappa}(\Gamma)$. If $f|_{\kappa}\gamma (z)=o(1)$ as  $z\to i\infty$ for all $\gamma\in\SL_2(\Z)$, then we call $f$ a \begin{it}cusp form\end{it}, and we let $S_{\kappa}(\Gamma)$ denote the space of cusp forms of weight $\kappa$. 

For $m\in\Z\setminus\{0\}$ and $\kappa>2$, the \emph{$m$-th Poincar\'e series of exponential type} for $\Gamma_0(4)$ is defined by (with $q:=e^{2\pi i z}$ throughout)
\[
P_{\kappa,m}(z):=\sum_{\gamma\in\Gamma_\infty\setminus\Gamma_0(4)} q^m\big|_{\kappa} \gamma,
\]
where $\Gamma_\infty := \{\pm\begin{psmallmatrix}1&n\\0&1\end{psmallmatrix}:n\in\Z\}$. For $\kappa=k+\frac{1}{2}\in\Z+\frac{1}{2}$ and $\Gamma\subseteq\Gamma_0(4)$, \begin{it}Kohnen's plus space\end{it} is the subspace of $M_{\kappa}(\Gamma)$ consisting of functions $f$ with Fourier expansions of the shape $f(z)=\sum_{(-1)^kD\equiv 0,1\pmod{4}} c(n)q^n$.
We let $|\operatorname{pr}$ denote the projection operator to the plus space (see \cite[Section 3]{KohnenNewforms}) and define the \emph{projected Poincar\'e series} 
\begin{equation}\label{eqn:Poincaredef}
P_{k+\frac{1}{2},m}^+:=P_{k+\frac{1}{2},m}|\operatorname{pr}.
\end{equation}

We next define local weak Maass forms. 

\begin{defn*}
	A function $f\colon\H\to\C$ which is real-analytic outside of an exceptional set $\mathcal{E}\subset\H$ of measure zero is a {\it local weak Maass form of weight $k\in\Z$, exceptional set $\mathcal{E}$, and eigenvalue $\lambda$ on} $\SL_2(\Z)$ if the following conditions hold: 
	\begin{enumerate}[leftmargin=*]
		\item We have, for $\begin{psmallmatrix}a&b\\c&d\end{psmallmatrix} \in \SL_2(\Z)$,
		\[
		f\left(\frac{az+b}{cz+d}\right) = (cz+d)^k f(z).
		\]
		
		\item For $\tau\notin\mathcal{E}$, there exists a neighborhood of $\tau$, in which we have
		\[
		\Delta_k(f)(z) = \lambda f(z).
		\]
		Here the {\it weight $k$ Laplace operator} is defined as
		\[
		\Delta_k := -y^2\left(\frac{\partial^2}{\partial x^2} + \frac{\partial^2}{\partial y^2}\right) + iky\left(\frac{\partial}{\partial x} + i\frac{\partial}{\partial y}\right).
		\]
		Note that
		\[
		\Delta_k = -\xi_{2-k}\circ\xi_k.
		\]
		
		\item The function $f$ has at most linear exponential growth at $i\infty$.
	\end{enumerate}
	If $\mathcal{E}=\emptyset$, then we simply call $f$ a \begin{it}weak Maass form\end{it} and if moreover $\lambda=0$, then $f$ is called a {\it harmonic Maass form}.
\end{defn*}

We also require the \emph{Petersson inner product} defined for $f,g$ having weight $k\in\N+\frac{1}{2}$ on $\Gamma_0(4)$ (in case of absolute convergence)
\begin{equation}\label{Pet}
	\langle f,g\rangle = \frac{1}{\left[\SL_2(\Z):\Gamma_0(4)\right]} \int_{\mathcal F_4} f(z) \overline{g(z)} y^k \frac{dxdy}{y^2},
\end{equation}
where $\mathcal F_4$ is a fundamental domain for $\Gamma_0(4)$.

\subsection{The raising and the lowering operator} 
We first recall a simple relation between modularity of weight $\kappa$ and $-\kappa$.

\begin{lem}\label{lem:conjmodular}
If $f$ satisfies weight $\kappa\in\Z$ modularity, then $y^{\kappa}\overline{f(z)}$ satisfies weight $-\kappa$ modularity. 
\end{lem}

Recall that the lowering operator $L_{\kappa}$ (resp. the raising operator $R_{\kappa}$) lowers (resp. raises) the weight and changes the eigenvalue of a weak Maass form (see \cite[Lemma 5.2 ii)]{HarmonicBook}).
\begin{lem}\label{lem:lower}
Suppose that $f$ is a weak Maass form of weight $\kappa\in\Z$ with eigenvalue $\lambda$. Then $L_{\kappa}(f)$ is a weak Maass form of weight $\kappa-2$ and eigenvalue $\lambda-\kappa+2$ and $R_{\kappa}(f)$ is a weak Maass form of weight $\kappa+2$ and eigenvalue $\lambda+\kappa$.
\end{lem}
The lowering and $\xi$-operators are related to each other.

\begin{lem}\label{lem:xilower}
If $F$ is real-analytic in some neighborhood of $z$, then we have
\[
y^{2-\kappa}\overline{\xi_\kappa(F(z))}=L_{\kappa}(F(z)).
\]
\end{lem}

The raising and the $\xi$-operator are also related by the following lemma.
\begin{lem}\label{lem:raisexi}
For a function $f$ which is continuously real-differentiable in a neighborhood of $z\in\H$, we have
\[
 R_{-\kappa}\left(y^{\kappa}\overline{f(z)}\right)=\xi_{\kappa}(f(z)).
\]
\end{lem}
Combining Lemmas \ref{lem:xilower} and \ref{lem:raisexi} yields the  following relation.
\begin{lem}\label{lem:Deltaconj}
If $f$ satisfies $\Delta_{\kappa}(f)=\lambda f$ in a neighborhood of $z\in\H$, then 
\[
\Delta_{-\kappa}\left(y^{\kappa}\overline{f(z)}\right)=\left(\overline{\lambda}+\kappa\right) y^{\kappa} \overline{f(z)}.
\]
 \end{lem}

We next compute the action of the raising operator on functions attached to binary quadratic forms.
\begin{lem}\label{lem:Qdiff}\hspace{0cm}
\noindent

\noindent	
\begin{enumerate}[leftmargin=*]
\item We have 
\[
R_{0}\left(\frac{Dy^2}{|Q(z,1)|^2}\right)=\frac{2Dy^2Q_z}{Q(\overline{z},1)Q(z,1)^2}.
\]
\item We have 
\[
R_{0}\left(\frac{Q_z^2}{D}\right)=R_{0,z}\left(\frac{|Q(z,1)|^2}{Dy^2}\right)=-\frac{2Q(\overline{z},1) Q_z}{Dy^2}.
\]
\end{enumerate}
\end{lem}
\begin{proof}
(1) The claim follows by a direct calculation using 
\begin{equation}\label{eqn:Qzeval}
\frac{Q(z,1)}{y} - iQ'(z,1) = Q_z.
\end{equation}

\noindent
(2) The second identity in (2) follows from (1). Using \cite[(3.3)]{BKV}, we have
\begin{equation}\label{eqn:Qz2}
\frac{Q_z^2}{D}=\frac{|Q(z,1)|^2}{Dy^2}-1.
\end{equation}
This yields the first identity.
\end{proof}

We also require the following lemma.

\begin{lem}\label{lem:fkDraise}\hspace{0cm}
	For $D\in\Z$ and $Q\in\mathcal{Q}_D$, we have 
	\[
	R_{2k}\left(\frac{1}{Q(z,1)^{k}}\right)=\frac{2k Q_z}{Q(z,1)^{k+1}}.
	\]
\end{lem}
\begin{proof}
	A direct computation gives
	\begin{equation}
		R_{2k}\left(\frac{1}{Q(z,1)^{k}}\right)=\frac{2k}{Q(z,1)^{k+1}}\left(\frac{Q(z,1)}{y} - iQ'(z,1)\right).\label{Raising the sum}
	\end{equation}
	Using \eqref{eqn:Qzeval} then gives the claim.
\end{proof}

\subsection{Lifts} We finally recall an interpretation of the quadratic form Poincar\'e series as lifts. The functions considered in this paper naturally appear as lifts of half-integral weight objects via the Shimura correspondence (see \cite{BKV,KohnenZagier}). In particular, we recall the construction of $f_{k,D}$ due to Kohnen and Zagier. For fixed $z\in\H$, they \cite[Theorem 2]{KohnenZagier} showed that 
\begin{equation}\label{eqn:Omegadef}
\Omega_k(\tau,z):=\sum_{D\ge 1} D^{k-\frac{1}{2}} f_{k,D}(z) e^{2\pi i D\tau}
\end{equation}
is a weight $k+\frac{1}{2}$ cusp form on $\Gamma_0(4)$ as a function of $\tau$. Moreover, we recover $f_{k,D}$ by integrating the Poincar\'e series $P_{k+\frac{1}{2}, D}(z)$ against $\Omega$:
\begin{equation} \label{eqn:fkDthetalift}
\left\langle \Omega_k(\cdot,z),P_{k+\frac{1}{2},D}^+\right\rangle=\frac{\Gamma\left(k-\frac{1}{2}\right)}{6(4\pi)^{k-\frac{1}{2}}} f_{k,D}(z). 
\end{equation}
	
\section{Proof of Theorems \ref{prop:fkDraise} and \ref{prop:omegaNegProp} and a reproof of Theorem \ref{thm:MRS}}\label{sec:reproof}

Theorem \ref{prop:fkDraise} follows immediately from Lemma~\ref{lem:fkDraise}.

\begin{proof}[Proof of Theorem \ref{prop:fkDraise}]
The series defining $f_{k,D}$ converges absolutely locally uniformly, so applying Lemma \ref{lem:fkDraise} termwise to \eqref{fk} yields the claim.
\end{proof}

To prove Theorem \ref{thm:MRS} we require the following symmetry property which follows from the fact that $[a,-b,c]$ runs through  $\mathcal Q_D$ if $[a,b,c]$ does.
\begin{lem}\label{lem:omegaflip}
We have 
\[
\overline{\omega_{k+1,D}(-\overline{z})}=\omega_{k+1,D}(z).
\]
\end{lem}
We are now ready to prove Theorem \ref{thm:MRS}.
\begin{proof}[Proof of Theorem \ref{thm:MRS}]
(1) This is a direct consequence of Theorem \ref{prop:fkDraise} and Lemma \ref{lem:lower}.\\
(2) The claim follows by Theorem \ref{prop:fkDraise}, \eqref{Raising the sum}, and \eqref{fk}.\\
(3) Bruinier, Kohnen, and Ono
 \cite[Theorem 1.1]{BKO} showed that, for a weight $2k$ modular form $f$, we have 
\[
\frac{1}{2\pi i} f'=\frac{k}{6}E_2 f -f f^{\operatorname{div}}.
\]
By Theorem \ref{prop:fkDraise}, we have
$$
\omega_{k+1,D}(z) = \frac{i}{k}f_{k,D}'(z)+\frac{f_{k,D}(z)}{y}.
$$ 
From this we conclude the proof of part (3).
\noindent

\noindent
(4) Applying raising termwise to \eqref{eqn:Omegadef}, we obtain, from Theorem \ref{prop:fkDraise}, that 
\begin{equation}\label{eqn:LambdaRaise}
R_{2k,z}\left(\Omega_k(\tau,z)\right)=2k\sum_{D\ge 1}D^{k-\frac{1}{2}}\omega_{k+1,D}(z)e^{2\pi i D\tau}=2k\Lambda_k(\tau,z)
\end{equation}
 satisfies weight $2k+2$ modularity on $\SL_2(\Z)$ in $z$ and weight $k+\frac{1}{2}$ modularity on $\Gamma_0(4)$ in $\tau$, as noted in  \cite[Theorem 1.2]{MRS}. 

Applying raising to  \eqref{eqn:fkDthetalift} and using Theorem \ref{prop:fkDraise} and \eqref{eqn:LambdaRaise} yields 
\begin{equation}
\left\langle \Lambda_k(\cdot,z),P_{k+\frac{1}{2},D}^+\right\rangle=\frac{\Gamma\left(k-\frac{1}{2}\right)}{6(4\pi)^{k-\frac{1}{2}}} \omega_{k+1,D}(z).\label{eqn:lift}
\end{equation}
Since the Petersson inner product is conjugate-symmetric, \eqref{eqn:lift} implies that 

\[
\left\langle P_{k+\frac{1}{2},D}^+,\Lambda_k(\cdot,-\overline{z})\right\rangle = \frac{\Gamma\left(k-\frac{1}{2}\right)}{6(4\pi)^{k-\frac{1}{2}}}\overline{\omega_{k+1,D}(-\overline{z})}.
\]
The claim now follows by Lemma \ref{lem:omegaflip}.
\end{proof}

We next prove Theorem \ref{prop:omegaNegProp}.

\begin{proof}[Proof of Theorem \ref{prop:omegaNegProp}]
 
(1) Using Lemma \ref{lem:xilower} and then \eqref{Ver} and \eqref{eqn:Wdef}, for $z\notin E_D$, we obtain that
\[
L_{2-2k}\left(\mathcal{F}_{1-k,D}(z)\right)=y^{2k}\overline{\xi_{2-2k}\left(\mathcal{F}_{1-k,D}(z)\right)}=\frac{y^{2k}\overline{f_{k,D}(z)}}{\binom{2k-2}{k-1}\pi}=\frac{\mathcal{W}_{-k,D}(z)}{\binom{2k-2}{k-1}\pi}.
\]
This gives the claimed identity for $\mathcal{W}_{-k,D}$. Since $f_{k,D}$ is a weight $2k$ cusp form by \cite[Theorem in Appendix 2]{ZagierDoiNaganuma}, the modularity of $\mathcal{W}_{-k,D}$ follows immediately from Lemma \ref{lem:conjmodular}. Since $f_{k,D}$ is annihilated by $\Delta_{2k}$, the claimed eigenvalue follows immediately from Lemma \ref{lem:Deltaconj}.\\
(2) By Theorem \ref{prop:fkDraise}, \eqref{Ver}, and Lemma \ref{lem:xilower}, for every $z\notin E_D$ we have
\begin{align*}
\omega_{k+1,D}(z)=\frac{1}{2k}\binom{2k-2}{k-1}\pi R_{2k}\left(y^{2k} \overline{L_{2-2k}\left(\mathcal{F}_{1-k,D}(z)\right)}\right).
\end{align*}
Substituting part (1) into the previous identity we obtain, for $z \notin E_D$, 
\begin{equation}\label{eqn:omegaRW}
\omega_{k+1,D}(z)=\frac{1}{2k} R_{2k}\left(y^{2k} \overline{\mathcal{W}_{-k,D}(z)}\right).
\end{equation}
However, since both sides of \eqref{eqn:omegaRW} extend continuously across $E_D$, \eqref{eqn:omegaRW}  holds for all $z\in\H$. The claim now follows by Lemma \ref{lem:raisexi}.
\end{proof}

\section{Other eigenvalues and the proof of Theorems \ref{thm:raiselowersparam} and \ref{prop:omegasparam}}\label{sec:sparam}
We now recall the spectral families introduced in \cite{BKV}. For $k\in\Z$, the authors and Viazovska \cite[(1.3)]{BKV}, defined, for $\re(s)\gg 1$,
\begin{equation}\label{eqn:fksDdef}
f_{k,s,D}(z):=\sum_{Q\in\mathcal{Q}_D}Q(z,1)^{-k}\varphi_{k,s}\left(\frac{Dy^2}{|Q(z,1)|^2}\right),
\end{equation}
where\footnote{We make the dependence on $k$ explicit since the weights vary throughout.}
\begin{equation}\label{eqn:phisdef}
\!\varphi_{k,s}(w):=\frac{\Gamma\!\left(s\!+\!\frac{k}{2}\!-\!\frac{1}{4}\right)\!D^{\frac{k}{2}+\frac{1}{4}}}{6\Gamma(2s)(4\pi)^{\frac{k}{2}-\frac{1}{4}}} w^{s-\frac{k}{2}-\frac{1}{4}} {_2F_1}\!\left(s+\frac{k}{2}-\frac{1}{4},s-\frac{k}{2}-\frac{1}{4};2s;w\right)\!.\!
\end{equation}
Here ${_2F_1}$ is Gauss's hypergeometric function. Following \cite[(1.7)]{BKV}, for $\ell\in\Z$ and $\re(s)\gg1$ we furthermore define (with $\ell=1-k$ there) 
\begin{equation}\label{eqn:calFsDdef}
\mathcal{F}_{\ell,s,D}(z):=\sum_{Q\in\mathcal{Q}_D}\frac{\sgn\left(Q_z\right)}{Q(z,1)^{\ell}}\psi_{\ell,s}\left(\frac{Dy^2}{|Q(z,1)|^2}\right)
\end{equation}
with 
\begin{equation}\label{eqn:phi*sdef}
\psi_{\ell,s}(w):=\frac{\Gamma\left(s-\frac{\ell}{2}+\frac{1}{4}\right)(4\pi D)^{\frac{\ell}{2}+\frac{1}{4}}}{12\sqrt{\pi}\Gamma(2s) w^{\frac{\ell}{2}+\frac{1}{4}-s}}{_2F_1}\left(s+\frac{\ell}{2}-\frac{1}{4},s-\frac{\ell}{2}-\frac{1}{4};2s;w\right).
\end{equation}
 For $k\in 2\N$, the functions $f_{k,s,D}$ and $\mathcal{F}_{1-k,s,D}$ were shown to be local Maass forms in  \cite[Theorem 1.1 (1) and Theorem 1.3 (1)]{BKV}.
\begin{lem}\label{lem:fFlocalMaass}
If $k\in 2\N$, then $f_{k,s,D}$ (resp. $\mathcal{F}_{1-k,s,D}$) is a local Maass forms of weight $2k$ (resp. $2-2k$) with exceptional set $E_D$ and eigenvalue $4\lambda_{k,s}$, where
\begin{equation}\label{eqn:lambdadef}
\lambda_{k,s}:=\left(s-\frac{k}{2}-\frac{1}{4}\right)\left(1-s-\frac{k}{2}-\frac{1}{4}\right).
\end{equation}
\end{lem}
We require the following relations to investigate the properties of $f_{k,s,D}$ (resp. $\mathcal{F}_{1-k,s,D}$) for $k\in -2\N$.

\begin{lem}\label{lem:flipweight}
\noindent

\noindent
\begin{enumerate}[leftmargin=*]
\item For $k\in\Z$ and $\re(s) \gg 1$,   we have 
\begin{align*}
f_{-k,s,D}(z)&=\frac{\Gamma\left(s-\frac{k}{2}-\frac{1}{4}\right) }{\Gamma\left(s+\frac{k}{2}-\frac{1}{4}\right)} (4\pi)^k y^{2k} f_{k,s,D}\left(-\overline{z}\right),\\
\mathcal{F}_{-k,s,D}(z)&=\frac{\Gamma\left(s+\frac{k}{2}+\frac{1}{4}\right) }{\Gamma\left(s-\frac{k}{2}+\frac{1}{4}\right)}(4\pi)^{-k} y^{2k} \mathcal{F}_{k,s,D}\left(-\overline{z}\right).
\end{align*}
\item For $k\in\Z$ and $\re(s)\gg1$, we have  
\[
\overline{f_{k,\overline{s},D}\left(-\overline{z}\right)}=f_{k,s,D}(z), \qquad \overline{\mathcal{F}_{k,\overline{s},D}\left(-\overline{z}\right)}=\mathcal{F}_{k,s,D}(z).
\]
\item For $k\in\Z\setminus\{0,1\}$ and $\re(s)\gg1$, the functions $f_{k,s,D}$ and $\mathcal{F}_{k,s,D}$ are local Maass forms of weight $2k$ with eigenvalue  $4\lambda_{k,s}$ and exceptional set $E_D$. We have $f_{k,s,D}=0$ if $k$ is odd and $\mathcal{F}_{k,s,D}=0$ if $k$ is even. 
\end{enumerate}
\end{lem}
\begin{proof}
(1) First note that the ${_2F_1}$-factor in \eqref{eqn:phisdef} is invariant under $k\mapsto -k$, since ${_2F_1}(a,b;c;w)={_2F_1}(b,a;c;w)$. Hence we obtain 
\[
\varphi_{-k,s}(w)= \frac{\Gamma\left(s-\frac{k}{2}-\frac{1}{4}\right)\left(4\pi w\right)^k }{\Gamma\left(s+\frac{k}{2}-\frac{1}{4}\right)D^k} \varphi_{k,s}(w).
\]
Substituting this into \eqref{eqn:fksDdef} and applying the change of variables $Q=[a,b,c]\mapsto Q^*:=[a,-b,c]$, which preserves $Q_D$, gives the first identity in (1). 

The same argument applied to \eqref{eqn:phi*sdef} gives
\[
\psi_{-k,s}(w)= \frac{\Gamma\left(s+\frac{k}{2}+\frac{1}{4}\right) w^{k}}{\Gamma\left(s-\frac{k}{2}+\frac{1}{4}\right) (4\pi D)^k} \psi_{k,s}(w).
\]
The second claim then follows by plugging into \eqref{eqn:calFsDdef}.

\noindent
(2) Note that $\overline{\varphi_{k,\overline{s}}}(w)=\varphi_{k,s}(w)$ for $w\in \R^+$. Plugging this into \eqref{eqn:fksDdef}, we have 
\[
\overline{f_{k,\overline{s},D}(-\overline{z})}=\sum_{Q\in\mathcal{Q}_D}\overline{Q(-\overline{z},1)}^{-k}\varphi_{k,s}\left(\frac{Dy^2}{|Q(-\overline{z},1)|^2}\right).
\]
Applying the same change of variables $Q\mapsto Q^*$ gives the first identity in (2). Similarly, using $\overline{\psi_{k,\overline{s}}(w)} = \psi_{k,s}(w)$ yields the second identity in (2).

\noindent
(3) The fact that $f_{k,s,D}=0$ if $k$ is odd and $\mathcal{F}_{k,s,D}=0$ if $k$ is even follows from letting  $Q\mapsto -Q$. It hence remains to show the claim for $f_{k,s,D}$ with $k\in 2\Z\setminus\{0\}$ and the claim for $\mathcal{F}_{k,s,D}$ with $k\in(1+2\Z)\setminus\{1\}$.
Lemma \ref{lem:fFlocalMaass} gives the claim for $f_{k,s,D}$ with $k\in 2\N$.  Using parts (1) and then (2), we next write
\[
f_{-k,s,D}(z)=\frac{\Gamma\left(s-\frac{k}{2}-\frac{1}{4}\right) }{\Gamma\left(s+\frac{k}{2}-\frac{1}{4}\right)} (4\pi)^k y^{2k}\overline{f_{k,\overline{s},D}(z)}.
\]
Since $f_{k,\overline{s},D}$ satisfies weight $2k$ modularity by \cite[Theorem 1.1 (1)]{BKV}, Lemma~\ref{lem:conjmodular} implies that the right-hand side is modular of weight $-2k$. Moreover, since $f_{k,\overline{s},D}$ has eigenvalue $4\lambda_{k,\overline{s}}$ away from $E_D$ by Lemma \ref{lem:fFlocalMaass}, we see from Lemma \ref{lem:Deltaconj} that $f_{-k,s,D}$ has eigenvalue $4\overline{\lambda_{k,\overline{s}}} + 2k$ away from $E_D$. Noting that $\overline{\lambda_{k,\overline{s}}}=\lambda_{k,s}$, the claim for $f_{-k,s,D}$ follows because 
\begin{equation}\label{eqn:lambda-k,s}
4\lambda_{k,s}+2k = 4\lambda_{-k,s}.
\end{equation}

The claim for $\mathcal{F}_{k,s,D}$ for $k\in 1-2\N$ follows by Lemma \ref{lem:fFlocalMaass} and the symmetry
\begin{equation}\label{eqn:lambdasymmetry}
\lambda_{1-k,s}=\lambda_{k,s}.
\end{equation}

It remains to treat $\mathcal{F}_{-k,s,D}$ (with $k\in 1-2\N$ as above). Using parts (1) and then (2), we have
\[
\mathcal{F}_{-k,s,D}(z)=\frac{\Gamma\left(s+\frac{k}{2}+\frac{1}{4}\right) }{\Gamma\left(s-\frac{k}{2}+\frac{1}{4}\right)} (4\pi)^{-k} y^{2k} \overline{\mathcal{F}_{k,\overline{s},D}(z)}.
\]
Since $k\in 1-2\N$, Lemma \ref{lem:fFlocalMaass} together with Lemma \ref{lem:conjmodular} implies that the right-hand side is modular of weight $-2k$ and Lemma \ref{lem:Deltaconj} implies that it has eigenvalue $4\overline{\lambda_{k,\overline{s}}} + 2k=4\lambda_{k,s}+2k$ under $\Delta_{-2k}$ outside of the exceptional set $E_D$. The claim for $\mathcal{F}_{-k,s,D}$ now follows by \eqref{eqn:lambda-k,s}.
\end{proof}

We are now ready to prove Theorem \ref{thm:raiselowersparam}.
 
\begin{proof}[Proof of Theorem \ref{thm:raiselowersparam}]
Assume throughout $z\notin E_D$, so that the relevant functions are real-analytic near $z$.\\ (1) Suppose that $\re(s)\gg1$ so \eqref{eqn:fksDdef} converges compactly. Then, since (in any neighborhood where $f$ and $g$ are real-analytic) 
\begin{equation}\label{Rfg}
	R_{2k}(fg)=gR_{2k}(f)+fR_0(g),	
\end{equation}
we have 
\begin{multline}\label{eqn:raisesparam}
R_{2k}\left(f_{k,s,D}(z)\right) = \sum_{Q\in\mathcal{Q}_D} R_{2k}\left(\frac{1}{Q(z,1)^{k}}\right)\varphi_{k,s}\left(\frac{Dy^2}{|Q(z,1)|^2}\right)\\
+ \sum_{Q\in\mathcal{Q}_D}Q(z,1)^{-k} R_0\left(\varphi_{k,s}\left(\frac{Dy^2}{|Q(z,1)|^2}\right)\right).
\end{multline}
We evaluate the two terms in \eqref{eqn:raisesparam} separately. By Lemma \ref{lem:fkDraise}, the first term in \eqref{eqn:raisesparam} is 
\begin{equation}\label{eqn:raisesparam-1st}
2k\sum_{Q\in\mathcal{Q}_D} \frac{Q_z}{Q(z,1)^{k+1}}\varphi_{k,s}\left(\frac{Dy^2}{|Q(z,1)|^2}\right).
\end{equation}

To evaluate the second term in \eqref{eqn:raisesparam}, we first rewrite
\[
R_0\left(\varphi_{k,s}\left(\frac{Dy^2}{|Q(z,1)|^2}\right)\right)=R_0\left(\frac{Dy^2}{|Q(z,1)|^2}\right)\varphi_{k,s}'\left(\frac{Dy^2}{|Q(z,1)|^2}\right). 
\]
Thus, using Lemma \ref{lem:Qdiff} (1), the second term in \eqref{eqn:raisesparam} equals 
\begin{equation}\label{eqn:raisesparam2nd}
2 \sum_{Q\in\mathcal{Q}_D}\frac{Q_z}{Q(z,1)^{k+1}} \frac{Dy^2}{|Q(z,1)|^2}\varphi_{k,s}'\left(\frac{Dy^2}{|Q(z,1)|^2}\right).
\end{equation}
It remains to compute $\varphi_{k,s}'$. We have
\begin{equation}\label{eqn:phisprime}
w\varphi_{k,s}'(w) = \frac{\Gamma\!\left(s+\frac{k}{2}-\frac{1}{4}\right) D^{\frac{k}{2}+\frac{1}{4}}}{6\Gamma(2s)(4\pi)^{\frac{k}{2}-\frac{1}{4}}} \left[w \frac{d}{dw} \left(w^b {}_2F_1(a,b;2s;w)\right)\right]_{\substack{a=s+\frac{k}{2}-\frac{1}{4}\\ b=s-\frac{k}{2}-\frac{1}{4}}}.
\end{equation}
By \cite[15.5.3]{NIST} with $n=1$, we have
\begin{equation*}
	w \frac{d}{dw} \left(w^a {}_2F_1(a,b;c;w)\right) = a w^a {}_2 F_1(a+1,b;c;w).
\end{equation*}
Noting that ${}_2F_1(a,b;c;w) = {}_2F_1(b,a;c;w)$, we thus obtain 
\begin{multline}\label{eqn:xdx2F1}
	\left[w\frac{d}{dw}\left(w^b {}_2F_1(a,b;2s;w)\right)\right]_{\substack{a=s+\frac{k}{2}-\frac{1}{4} \\ b=s-\frac{k}{2}-\frac{1}{4}}}  \\
	= \left(s-\frac{k}{2}-\frac{1}{4}\right) w^{s-\frac{k}{2}-\frac{1}{4}} {}_2F_1\!\left(s+\frac{k}{2}-\frac{1}{4},s-\frac{k}{2}+\frac{3}{4};2s;w\right).
\end{multline}
Therefore, \eqref{eqn:phisprime} becomes 
\begin{multline*}
 \frac{\left(s\!-\!\frac{k}{2}\!-\!\frac{1}{4}\right)\Gamma\!\left(s\!+\!\frac{k}{2}\!-\!\frac{1}{4}\right) D^{\frac{k}{2}+\frac{1}{4}}}{6\Gamma(2s)(4\pi)^{\frac{k}{2}-\frac{1}{4}}} w^{s-\frac{k}{2}-\frac{1}{4}} {}_2F_1\!\left(s+\frac{k}{2}-\frac{1}{4},s-\frac{k}{2}+\frac{3}{4};2s;w\right)\!.\!\!\!\!
\end{multline*}
Plugging this into (\ref{eqn:phisprime}) and then into (\ref{eqn:raisesparam2nd}) yields that (\ref{eqn:raisesparam2nd}) equals
\begin{multline*}
\frac{\left(s-\frac{k}{2}-\frac{1}{4}\right)\Gamma\!\left(s+\frac{k}{2}-\frac{1}{4}\right) D^{\frac{k}{2}+\frac{1}{4}}}{3\Gamma(2s)(4\pi)^{\frac{k}{2}-\frac{1}{4}}} \sum_{Q\in\mathcal{Q}_D}\frac{Q_z}{Q(z,1)^{k+1}} \left(\frac{Dy^2}{|Q(z,1)|^2}\right)^{s-\frac{k}{2}-\frac{1}{4}}\\
\times {}_2F_1\!\left(s+\frac{k}{2}-\frac{1}{4},s-\frac{k}{2}+\frac{3}{4};2s;\frac{Dy^2}{|Q(z,1)|^2}\right).
\end{multline*}
Inserting this and \eqref{eqn:raisesparam-1st} into \eqref{eqn:raisesparam} then yields 
\begin{align}
&R_{2k}\left(f_{k,s,D}(z)\right)\nonumber\\
&\ \ =\sum_{Q\in\mathcal{Q}_D} \frac{Q_z}{Q(z,1)^{k+1}}\Bigg[2k\varphi_{k,s}\left(\frac{Dy^2}{|Q(z,1)|^2}\right)+ \frac{\left(s-\frac{k}{2}-\frac{1}{4}\right)\Gamma\!\left(s+\frac{k}{2}-\frac{1}{4}\right) D^{\frac{k}{2}+\frac{1}{4}}}{3\Gamma(2s)(4\pi)^{\frac{k}{2}-\frac{1}{4}}}\nonumber\\
&\quad\ \times \left(\frac{Dy^2}{|Q(z,1)|^2}\right)^{s-\frac{k}{2}-\frac{1}{4}} {}_2F_1\!\left(s+\frac{k}{2}-\frac{1}{4},s-\frac{k}{2}+\frac{3}{4};2s;\frac{Dy^2}{|Q(z,1)|^2}\right)\Bigg].\!\label{eqn:raisesparam-final}
\end{align}
We now rewrite the bracketed term so the resulting summand matches the definition of $\mathcal{F}_{k+1,s,D}$. Using \eqref{eqn:phisdef}, the terms inside the  brackets in \eqref{eqn:raisesparam-final} become 
\begin{align}
&\!\!\frac{\Gamma\left(s+\frac{k}{2}-\frac{1}{4}\right)D^{\frac{k}{2}+\frac{1}{4}}}{6\Gamma(2s)(4\pi)^{\frac{k}{2}-\frac{1}{4}}} \left(\frac{Dy^2}{|Q(z,1)|^2}\right)^{s-\frac{k}{2}-\frac{1}{4}}\nonumber\\
&\ \times\Bigg(2k {_2F_1}\left(s+\frac{k}{2}-\frac{1}{4},s-\frac{k}{2}-\frac{1}{4};2s;\frac{Dy^2}{|Q(z,1)|^2}\right)\nonumber\\
&\qquad\ +2 \left(s-\frac{k}{2}-\frac{1}{4}\right) {}_2F_1\!\left(s+\frac{k}{2}-\frac{1}{4},s-\frac{k}{2}+\frac{3}{4};2s;\frac{Dy^2}{|Q(z,1)|^2}\right)\Bigg).\label{eqn:raisesparam-expand}
\end{align}
We next recall \cite[15.5.12]{NIST}, which can be rewritten as 
\begin{equation*}
(a-b){_2F_1}(a,b;c;z)=a{_2F_1}(a+1,b;c;z)- b{_2F_1}(a,b+1;c;z). 
\end{equation*}
Using this with $a=s+\frac{k}{2}-\frac{1}{4}$, $b=s-\frac{k}{2}-\frac{1}{4}$, and $c=2s$, we obtain
\begin{multline}\label{rewrite1}
k {_2F_1}\left(s+\frac{k}{2}-\frac{1}{4},s-\frac{k}{2}-\frac{1}{4};2s;\frac{Dy^2}{|Q(z,1)|^2}\right)\\
=\left(s+\frac{k}{2}-\frac{1}{4}\right){_2F_1}\left(s+\frac{k}{2}+\frac{3}{4},s-\frac{k}{2}-\frac{1}{4};2s;\frac{Dy^2}{|Q(z,1)|^2}\right)\\
-\left(s-\frac{k}{2}-\frac{1}{4}\right){_2F_1}\left(s+\frac{k}{2}-\frac{1}{4},s-\frac{k}{2}+\frac{3}{4};2s;\frac{Dy^2}{|Q(z,1)|^2}\right).
\end{multline}
Inserting this into \eqref{eqn:raisesparam-expand}, using that $s\Gamma(s)=\Gamma(s+1)$, and then plugging
into \eqref{eqn:raisesparam-final}, we obtain 
\begin{align}
R_{2k}\left(f_{k,s,D}(z)\right)&=\frac{\Gamma\left(s+\frac{k}{2}+\frac{3}{4}\right)D^{\frac{k}{2}+\frac{1}{4}}}{3\Gamma(2s)(4\pi)^{\frac{k}{2}-\frac{1}{4}}}\sum_{Q\in\mathcal{Q}_D} \frac{Q_z}{Q(z,1)^{k+1}} \left(\frac{Dy^2}{|Q(z,1)|^2}\right)^{s-\frac{k}{2}-\frac{1}{4}}\nonumber\\
&\qquad\qquad\times{_2F_1}\left(s+\frac{k}{2}+\frac{3}{4},s-\frac{k}{2}-\frac{1}{4};2s;\frac{Dy^2}{|Q(z,1)|^2}\right)\nonumber \\
&=\frac{\Gamma\left(s+\frac{k}{2}+\frac{3}{4}\right)D^{\frac{k}{2}+\frac{3}{4}}}{3\Gamma(2s)(4\pi)^{\frac{k}{2}-\frac{1}{4}}}\sum_{Q\in\mathcal{Q}_D} \frac{\sgn\left(Q_z\right)}{Q(z,1)^{k+1}}\sqrt{1-\frac{Dy^2}{|Q(z,1)|^2}}\nonumber\\
&\hspace{-1.85cm}\times\left(\frac{Dy^2}{|Q(z,1)|^2}\right)^{s-\frac{k}{2}-\frac{3}{4}}{_2F_1}\left(s+\frac{k}{2}+\frac{3}{4},s-\frac{k}{2}-\frac{1}{4};2s;\frac{Dy^2}{|Q(z,1)|^2}\right),\label{eqn:raisesparam-final3}
\end{align}
using that, by \eqref{eqn:Qz2}, we have
\begin{equation}\label{eqn:Qz2-2}
	Q_z=\sqrt{D} \sgn\left(Q_z\right)\frac{\sqrt{1-\frac{Dy^2}{|Q(z,1)|^2}}}{\sqrt{\frac{Dy^2}{|Q(z,1)|^2}}}.
\end{equation}
It remains to put the hypergeometric factor into the form occurring in $\mathcal{F}_{k+1,s,D}$. We use Euler's transformation \cite[15.8.1]{NIST} 
\begin{equation}\label{eqn:Euler2F1}
	{}_2F_1(a,b;c;w) =(1-w)^{c-a-b} {}_2 F_1(c-a,c-b;c;w).
\end{equation}
Thus \eqref{eqn:raisesparam-final3} equals
\begin{multline*}
R_{2k}\left(f_{k,s,D}(z)\right)=\frac{\Gamma\left(s+\frac{k}{2}+\frac{3}{4}\right)D^{\frac{k}{2}+\frac{3}{4}}}{3\Gamma(2s)(4\pi)^{\frac{k}{2}-\frac{1}{4}}}\sum_{Q\in\mathcal{Q}_D} \frac{\sgn\left(Q_z\right)}{Q(z,1)^{k+1}}\left(\frac{Dy^2}{|Q(z,1)|^2}\right)^{s-\frac{k}{2}-\frac{3}{4}}\\
\times{_2F_1}\left(s-\frac{k}{2}-\frac{3}{4},s+\frac{k}{2}+\frac{1}{4};2s;\frac{Dy^2}{|Q(z,1)|^2}\right).\!\!\!
\end{multline*}
Noting that ${}_2 F_1 (a,b;c;w)={}_2F_1 (b,a;c;w)$ then implies the first identity in (1).

The proof of the second identity in part (1) is parallel. Applying $R_{2k}$ to \eqref{eqn:calFsDdef}, we obtain
\begin{equation}\label{RF}
	R_{2k} \left(\mathcal{F}_{k,s,D} (z)\right) = \sum_{Q\in \mathcal{Q}_D} \sgn(Q_z) R_{2k} \left(\frac{\psi_{k,s} \left(\frac{Dy^2}{\left|Q(z,1)\right|^2}\right)}{Q(z,1)^k}\right).
\end{equation}
Using the product rule \eqref{Rfg} yields 
\begin{equation*}
	R_{2k} \!\left(\frac{\psi_{k,s} \! \left(\frac{Dy^2}{\left|Q(z,1)\right|^2}\right)}{Q(z,1)^k}\right) \! = R_{2k} \! \left(Q(z,1)^{-k}\right) \psi_{k,s} \! \left(\frac{Dy^2}{\left|Q(z,1)\right|^2}\right) + \frac{R_0 \! \left(\psi_{k,s} \! \left(\frac{Dy^2}{\left|Q(z,1)\right|^2}\right)\!\right)}{Q(z,1)^k} \!.
\end{equation*}
By Lemma \ref{lem:fkDraise}, the first term contributes to \eqref{RF}
\begin{equation*}
	2k \sum_{Q\in \mathcal{Q}_D} \frac{|Q_z|}{Q(z,1)^{k+1}} \psi_{k,s} \left(\frac{Dy^2}{\left|Q(z,1)\right|^2}\right). 
\end{equation*}

For the second term we proceed as below  \eqref{eqn:raisesparam-1st} with $\varphi_{k,s}$ replaced with $\psi_{k,s}$, yielding as contribution to \eqref{RF} 
\begin{equation*}
	2\sum_{Q\in\mathcal{Q}_D } \frac{\left|Q_z\right|}{Q(z,1)^{k+1}} \left[w\psi_{k,s}'(w)\right]_{w= \frac{Dy^2}{\left|Q(z,1)\right|^2}}.
\end{equation*}
Combining the two contributions and using \eqref{eqn:Qz2-2} and plugging into \eqref{RF}, we obtain
\begin{multline*}
	R_{2k}\left(\mathcal F_{k,s,D}(z)\right) \\
	= 2\sqrt{D} \sum_{Q\in\mathcal Q_D} \frac{1}{Q(z,1)^{k+1}} \left[\sqrt{\frac{1-w}{w}}\left(k\psi_{k,s}(w)+w\psi_{k,s}'(w)\right)\right]_{w=\frac{Dy^2}{|Q(z,1)|^2}}.
\end{multline*}

We next compute, using \eqref{eqn:phi*sdef} and \eqref{eqn:xdx2F1},
\begin{multline*}
	\!\!\!\!w\psi_{k,s}'(w) \\
	= \frac{\left(s\!-\!\frac{k}{2}\!-\!\frac{1}{4}\right)\!\Gamma\!\left(s\!-\!\frac{k}{2}\!+\!\frac{1}{4}\right)\!(4\pi D)^{\frac{k}{2} + \frac{1}{4}} }{12\sqrt{\pi} \Gamma(2s)}  w^{s-\frac{k}{2}-\frac{1}{4}} {}_2F_1 \!\left(s+\frac{k}{2}-\frac{1}{4}, s-\frac{k}{2}+ \frac{3}{4};2s;w\right)\!. 
\end{multline*}
Using \eqref{rewrite1} followed by \eqref{eqn:Euler2F1}, the hypergeometric factor becomes the one appearing in $f_{k+1,s,D}$. Comparing with ~\eqref{eqn:fksDdef} gives the second identity.

\noindent (2) By Lemma \ref{lem:flipweight} (3), for $k\in\Z\setminus\{0,1\}$, the functions $f_{k,s,D}$ and $\mathcal{F}_{k,s,D}$ have eigenvalue $4\lambda_{k,s}$ under $\Delta_{2k}$ (outside of $E_D$). Using \eqref{eqn:DeltaLR} with $-\kappa=2k$, we rewrite 
\begin{equation}\label{eqn:DeltafRewrite}
4\lambda_{k,s}f_{k,s,D}(z)=\Delta_{2k}(f_{k,s,D}(z))=-L_{2k+2}\circ R_{2k}\left(f_{k,s,D}(z) \right)-2k f_{k,s,D}(z).
\end{equation}
Substituting the first identity of part (1) into the first term on the right-hand side of \eqref{eqn:DeltafRewrite} and rearranging, we obtain  
\[
-\frac{2\Gamma\left(s+\frac{k}{2}+\frac{3}{4}\right)}{(4\pi)^k\Gamma\left(s-\frac{k}{2}-\frac{1}{4}\right)} L_{2k+2}\left(\mathcal{F}_{k+1,s,D}(z)\right) = \left(4\lambda_{k,s}+2k\right) f_{k,s,D}(z).
\]
Using \eqref{eqn:lambda-k,s} and \eqref{eqn:lambdadef}, we obtain 
\begin{equation*}
L_{2k+2}\left(\mathcal{F}_{k+1,s,D}(z)\right)= \frac{2(4\pi)^k\left(s-\frac{k}{2}-\frac{3}{4}\right)\Gamma\left(s-\frac{k}{2}-\frac{1}{4}\right)}{\Gamma\left(s+\frac{k}{2}-\frac{1}{4}\right)}f_{k,s,D}(z).
\end{equation*}
Letting $k\mapsto k-1$ gives the second identity. 
Similarly, \eqref{eqn:DeltaLR} implies that 
\begin{equation}\label{eqn:DeltaFRewrite}
4\lambda_{k,s}\mathcal{F}_{k,s,D}(z)\!=\!\Delta_{2k}\left(\mathcal{F}_{k,s,D}(z)\right)\!=\!-L_{2k+2}\circ R_{2k}\left(\mathcal{F}_{k,s,D}(z) \right)-2k \mathcal{F}_{k,s,D}(z).
\end{equation}
Plugging the second identity from (1) into \eqref{eqn:DeltaFRewrite} and rearranging, we have 
\[
-\frac{2(4\pi)^k\left(s+\frac{k}{2}-\frac{1}{4}\right)\Gamma\left(s-\frac{k}{2}+\frac{1}{4}\right)}{\Gamma\left(s+\frac{k}{2}+\frac{1}{4}\right)} L_{2k+2}\left(f_{k+1,s,D}(z)\right) = \left(4\lambda_{k,s}+2k\right) \mathcal{F}_{k,s,D}(z).
\]
Using \eqref{eqn:lambda-k,s}, \eqref{eqn:lambdasymmetry}, \eqref{eqn:lambdadef}, and $\Gamma(w+1)=w\Gamma(w)$ yields 
\begin{align*}
L_{2k+2}\!\left(f_{k+1,s,D}(z)\right)&= -\frac{2\Gamma\!\left(s\!+\!\frac{k}{2}\!+\!\frac{1}{4}\right)\lambda_{k+1,s}}{(4\pi)^k\left(s\!+\!\frac{k}{2}\!-\!\frac{1}{4}\right)\Gamma\!\left(s\!-\!\frac{k}{2}\!+\!\frac{1}{4}\right)}\mathcal{F}_{k,s,D}(z) \\
&= \frac{2\Gamma\!\left(s\!+\!\frac{k}{2}\!+\!\frac{1}{4}\right)}{(4\pi)^k\Gamma\!\left(s\!-\!\frac{k}{2}\!-\!\frac{3}{4}\right)}\mathcal{F}_{k,s,D}(z).
\end{align*}
Letting $k\mapsto k-1$ gives the first identity in (2). 
\end{proof}

We are now ready to prove Theorem \ref{prop:omegasparam}.

\begin{proof}[Proof of Theorem \ref{prop:omegasparam}]
Setting $\ell=k+1$ in \eqref{eqn:phi*sdef} and dividing by $\Gamma(s-\frac{k}{2}-\frac{1}{4})$, we obtain
\begin{equation}\label{Verweiss}
\lim_{s\to \frac{k}{2}+\frac{1}{4}}\frac{\psi_{k+1,s}(w)}{\Gamma\!\left(s-\frac{k}{2}-\frac{1}{4}\right)} = \frac{(4\pi D)^{\frac{3}{4}+\frac{k}{2}}}{12\sqrt{\pi}\Gamma\!\left(k+\frac{1}{2}\right) \sqrt{w}} {_2F_1}\!\left(k+\frac{1}{2},-\frac{1}{2};k+\frac{1}{2};w\right)\!.
\end{equation}
By \cite[15.4.6]{NIST}, we have 
\[
{_2F_1}\left(k+\frac{1}{2},-\frac{1}{2};k+\frac{1}{2};w\right)=\sqrt{1-w}.
\]
Thus \eqref{Verweiss} becomes, using $\Gamma\left(k+\frac{1}{2}\right)=\frac{\sqrt{\pi} (2k)!}{4^k k!}$ and \eqref{eqn:Qz2-2},
\begin{equation*}
\lim_{s\to \frac{k}{2}+\frac{1}{4}}\frac{\psi_{k+1,s}(w)}{\Gamma\left(s-\frac{k}{2}-\frac{1}{4}\right)}= \frac{(4\pi D)^{\frac{3}{4}+\frac{k}{2}}4^kk!}{12\pi(2k)!} \sqrt{\frac{1-w}{w}}.
\end{equation*}
Plugging in  $w=\frac{Dy^2}{|Q(z,1)|^2}$ to the right-hand side above and using \eqref{eqn:Qz2-2} then yields 
\[ 
\frac{4^kk!(4\pi D)^{\frac{3}{4}+\frac{k}{2}}\left|Q_z\right|}{12\pi(2k)!\sqrt{D}}.
\]
Substituting this into the defining identity \eqref{eqn:calFsDdef} completes the argument.
\end{proof}

\end{document}